\documentclass[12pt]{article}
\usepackage[latin1]{inputenc}
\usepackage{amsmath}
\usepackage{amsfonts}
\usepackage{amssymb}
\usepackage[american]{babel}
\usepackage{graphicx}
\usepackage{subcaption}
\usepackage{amsthm}
\usepackage{color}
\usepackage{xcolor}
\usepackage{url}
\usepackage[normalem]{ulem}
\usepackage{tikz}
\usepackage{hyperref}
\usetikzlibrary{arrows}
\usetikzlibrary{arrows.meta}
\usetikzlibrary{angles,quotes} % Librerie necessarie per angoli ed etichette

\definecolor{angleRed}{RGB}{255, 192, 192}   % Rosa/Rosso chiaro
\definecolor{angleGreen}{RGB}{220, 240, 220} % Verde chiaro

\begin{document}

\title{%Complex flows on flower snarks
On the $2$-dimensional flow number of the Flower snarks}{}

\newtheorem{theorem}{Theorem}
\newtheorem{proposition}[theorem]{Proposition}
\newtheorem{lemma}[theorem]{Lemma}
\newtheorem{definition}[theorem]{Definition}
\newtheorem{example}[theorem]{Example}
\newtheorem{corollary}[theorem]{Corollary}
\newtheorem{conjecture}[theorem]{Conjecture}
\newtheorem{remark}{Remark}
\newtheorem{problem}{Problem}
\newtheorem{claim}{Claim}

\author{Davide Mattiolo\footnote{Department of Computer Science, KU Leuven Kulak, 8500 Kortrijk, Belgium. Email: \href{mailto:davide.mattiolo@kuleuven.be}{davide.mattiolo@kuleuven.be}}, Jozef Rajn\'{i}k\footnote{Department of Computer Science, Comenius University in Bratislava, Slovakia, Email: \href{mailto:jozef.rajnik@fmph.uniba.sk}{jozef.rajnik@fmph.uniba.sk}}}
\maketitle

\begin{abstract}
    Let $r\ge 2$ be a real number, $d$ a positive integer. A $d$-dimensional nowhere-zero $r$-flow, or $(r,d)$-NZF, on a graph $G$ is an orientation of $G$ together with a function $f\colon E(G)\to \mathbb{R}^d$, such that for all $e\in E(G)$, the Euclidean norm of $f(e)$ lies in the interval $[1,r-1]$, and for every $v\in V(G)$ the sum of all incoming flow values at $v$ equals the sum of all outgoing ones. The $d$-dimensional flow number of $G$ is the parameter $\phi_d(G)=\inf \{r\colon G$ has an $(r,d)$-NZF$\}$. 
    
    In this paper we provide a lower bound for the $2$-dimensional flow number of the the Flower snark. In particular, together with a previous numerical result by the authors, we prove that $\phi_2(J_n) \in [1 + 2 \sin\frac{5}{22}\pi, 2.387893647]$, where $J_n$ denotes the Flower snark on $4n$ vertices.
\end{abstract}



\section{Introduction}\label{section intro}

Let $r\ge 2$ be a real number, $d$ a positive integer, and $G$ a graph. A {\em $d$-dimensional nowhere-zero $r$-flow}, or $(r,d)$-NZF, on $G$ is an orientation of $G$ together with a function $f\colon E(G)\to \mathbb{R}^d$, such that for all $e\in E(G)$, the Euclidean norm of $f(e)$ lies in the interval $[1,r-1]$, and for every $v\in V(G)$ the sum of all incoming flow values at $v$ equals the sum of all outgoing ones (Kirchhoff's law). For a graph $G$, the $d$-dimensional flow number is the parameter $\phi_d(G)=\inf \{r\colon G$ has an $(r,d)$-NZF$\}$. It is well-known that $\phi_d$ is the minimum for every bridgeless graph, see \cite{tarsizhang, MMRT_d-dim-flows}.

Vector flows have recently received attention as they generalize the classical notion of nowhere-zero flow. Indeed, an $(r,1)$-NZF corresponds to the well-known notion of a circular nowhere-zero $r$-flow. One of the major open problems in this area is the $5$-Flow Conjecture of Tutte \cite{tut}, which claims that every bridgeless cubic graph $G$ verifies $\phi_1(G)\le5$.
For $d\ge3$, a conjecture of Jain \cite{DeVos_unit_vector} claims that every bridgeless graph has a $(2,d)$-NZF.

In this paper, we focus on the case $d=2$.
It is shown in \cite{MMRT_d-dim-flows} that every bridgeless graph $G$ verifies $\phi_2(G)\le 1+\sqrt5$, and that the Oriented $5$-Cycle Double Cover Conjecture \cite{arc, Jaeger} implies $\phi_2(G)\le \tau^2$ for every bridgeless graph $G$, where $\tau$ denotes the Golden Ratio.

The parameter $\phi_2(G)$ is known for very few classes of graphs, and its computation seems very hard. In \cite{Tho}, Thomassen proved that a cubic graph $G$ is bipartite if and only if $\phi_2(G)=2$; and that a planar graph $G$ has $\phi_2(G)=2$ if and only if its dual
graph is homomorphic to a subgraph of the unit distance graph. Moreover, in \cite{MMRT_wheels}, the authors computed the $2$-dimensional flow number of wheel graphs and the prism graphs. Finally, in \cite{MMRT_d-dim-flows}, the authors proved that a $3$-edge-colorable cubic graph $G$ has $\phi_2(G)\le 1+\sqrt2$, where equality holds if $G$ contains a triangle. Further interesting investigations of unit vector flows are done in \cite{Zhangandal}.

Given that the $2$-dimensional flow number of a $3$-edge-colorable cubic graph already has an upper bound (which is in some cases sharp), it would be interesting to understand the behavior of this parameter for the class of snarks ($2$-connected cubic graphs non-admitting a $3$-edge-coloring). To our knowledge, there is no snark for which this parameter is precisely known. Let us mention that it is conjectured in \cite{MMRT_d-dim-flows} that the $2$-dimensional flow number of the Petersen graph is $1+\sqrt\frac{7}{3}$.

In this paper, we investigate the $2$-dimensional flow number of the Flower snarks, a family of snarks introduced by Isaacs \cite{Isaacs}. Let $n\ge3$ be an odd natural number.
Let $F_1,\dots,F_n$ be $n$ copies of the $6$-pole $F$ in Figure \ref{fig:flower construction}. By a $6$-pole, we mean a graph with $6$ \emph{dangling} edges, i.e.\ edges incident with only one vertex; for more formal background of this notion, we refer the reader to \cite{Fiol, Nedela}. We consider the labelings depicted in Figure \ref{fig:flower construction}.
For every $i\in \{1,\dots,n\}$, we denote the copy of the dangling edge $x \in \{e_1, e_2, e_3, e_1', e_2', e_3'\}$ in the $6$-pole $F_i$ by $x^i$.
The Flower snark $J_n$ is constructed by joining together these $6$-poles by performing the junction of the edge $(e_j')^i$ with the edge $(e_j)^{i+1}$ for each $i \in \{1, \dots, n\}$ and $j \in \{1, 2, 3\}$, where the indices are taken modulo $n$.

\begin{figure}
\begin{center}
    \begin{tikzpicture}[scale=1.5]
  % Coordinates
  \coordinate (A) at (5.01, 6.00);
  \coordinate (B) at (5.00, 7.00);
  \coordinate (C) at (4.25, 7.75);
  \coordinate (D) at (5.74, 7.75);
  \coordinate (E) at (4.00, 6.00);
  \coordinate (F) at (6.00, 6.00);
  \coordinate (G) at (7.16, 7.73);
  \coordinate (H) at (3.50, 7.74);
  \coordinate (I) at (3.49, 8.24);
  \coordinate (J) at (5.75, 8.25);
  \coordinate (K) at (4.00, 8.25);
  \coordinate (L) at (6.49, 8.26);
  \coordinate (M) at (4.25, 8.25);
  \coordinate (N) at (6.00, 8.25);
  \coordinate (O) at (6.49, 7.76);

  % Drawing
  \draw (B) -- (C);
  \draw (B) -- (D);
  \draw (A) -- (B);
  \draw (A) -- (E);
  \draw (A) -- (F);
  %\draw (D_1) -- (G);
  \draw (C) -- (H);
  \draw (D) .. controls (5.75, 8.25) and (4.00, 8.25) .. (I);
  \draw (C) .. controls (4.25, 8.25) and (6.00, 8.25) .. (L);
  \draw (D) -- (O);
  \fill[black] (A) circle (3pt);
  \fill[black] (B) circle (3pt);
  \fill[black] (C) circle (3pt);
  \fill[black] (D) circle (3pt);
  \node[left, black] at (E) {$e_3$};
  \node[right, black] at (F) {$e'_3$};
  \node[left, black] at (H) {$e_2$};
  \node[left, black] at (I) {$e_1$};
  \node[right, black] at (L) {$e'_1$};
  \node[right, black] at (O) {$e'_2$};
\end{tikzpicture}
\caption{The Flower $6$-pole.}\label{fig:flower construction}
\end{center}
\end{figure}
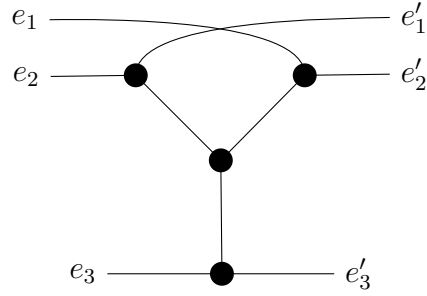

The main result of this paper is the following.

\begin{theorem}
    \label{thm:lower-bound}
    For every odd $n\ge3$, $\phi_2(J_n) \ge 1 + 2 \sin\frac{5}{22}\pi.$
\end{theorem}

We will also prove that the sequence $(\phi_2(J_{2k+1}))_{k\ge 1}$ is non-increasing.
Together with a numerical result of \cite{MMRT_d-dim-flows}, we then get that, for every odd $n\ge5,$ $ \phi_2(J_n) \in [1 + 2 \sin\frac{5}{22}\pi, 2.387893647].$

In \cite{LukSkov, Steffen}, it was proved that $\phi_1(J_n)=4+\frac1n$ for every odd $n\ge 3$. We suspect that the sequence $\phi_2(J_n)$ is also strictly decreasing and we leave the following open problem.

\begin{problem}
    Is it true that $\inf\{\phi_2(J_n)\colon n\ge3$ odd number$\} = 1 + 2 \sin\frac{5}{22}\pi$?
\end{problem}

\section{Notation and preliminary results}

Let $f$ be an $(r,2)$-NZF on $J_n$. In every $6$-pole $F_i$, we fix the orientation and flow values shown in Figure \ref{fig:Flower_pole}.

In what follows, given two vectors $u,v\in \mathbb{R}^2$, we denote by $\angle(u,v)$ the angle in the interval $(-\pi,\pi]$ needed to rotate the vector $u$ till it overlaps $v$, i.e.\ till they have the same direction and orientation. Note that $\angle(u,v)$ is positive if and only if the rotation is counterclockwise.
If $A$, $B$, and $C$ are points in $\mathbb{R}^2$, we denote $\angle(A - B, C - B)$ by $\angle ABC$.

\begin{lemma}
\label{lemma:two_angles}
Let $ABC$ and $BAD$ be two positively oriented triangles whose sides all have lengths in the interval $[1, \lambda]$  for some $1 \le \lambda < 2$. Then $|\angle BAC - \angle ABD| \le 3\arcsin \frac \lambda2 - \frac\pi2$.
\end{lemma}

\begin{proof}
Let $\varphi = \angle BAC$ and $\psi = \angle ABD$. It is sufficient to find the maximal value of $\varphi - \psi$, so we assume that $\varphi > \psi$. Firstly, we fix the length of $AB$ and determine the maximal value of $\varphi - \psi$. With fixed $AB$, the angle $\varphi$ reaches its maximum value for $AC = 1$ and $BC = \lambda$, while $\psi$ reaches its minimum value for $AD = 1$ and $BD = \lambda$. Then the triangles $ABC$ and $ABD$ are congruent, so it remains to find the maximum possible value of $\angle BAC - \angle ABC = \varphi - \psi$ among all triangles $ABC$ with fixed $BC = \lambda$ and $CA = 1$, and variable $AB \in [1, \lambda]$.

Let $E$ be the intersection points of $BC$ and the perpendicular bisector of $AB$. Then $\angle EAC = \angle BAC - \angle BAE = \varphi - \psi$. Also, $BE = AE$, so $AE + EC = \lambda$, thus, the point $E$ lies on the ellipse with focal points $A$ and $C$ and diameter $\lambda$. With decreasing $AB$,
the point $E$ moves along the ellipse, so the angle $EAC$ reaches the maximal value for $AB = 1$, which satisfies the triangle inequality since $1 \le \lambda < 2$.
% the angle $ACE$ decreases too and so $\angle EAC$ increases. So we obtain the maximal value of $\angle EAB = \varphi - \psi$ \dm{EAC?} for $AB = 1$.
Then $\varphi = 2 \arcsin \frac \lambda2$ and $\psi = \frac\pi2 - \frac\varphi2$. Therefore,
$$\varphi - \psi = \varphi - \frac\pi2 + \frac\varphi2 = \frac32 \varphi - \frac\pi2 = 3\arcsin \frac \lambda2 - \frac\pi2.$$\qedhere
\end{proof}

\begin{figure}
\begin{center}
    \begin{tikzpicture}[scale=1.5]

    every edge/.style = {draw, ->}

  % Coordinates
  \coordinate (A) at (5.01, 6.00);
  \coordinate (B) at (5.00, 7.00);
  \coordinate (C) at (4.25, 7.75);
  \coordinate (D) at (5.74, 7.75);
  \coordinate (E) at (4.00, 6.00);
  \coordinate (F) at (6.00, 6.00);
  \coordinate (G) at (7.16, 7.73);
  \coordinate (H) at (3.50, 7.74);
  \coordinate (I) at (3.49, 8.24);
  \coordinate (J) at (5.75, 8.25);
  \coordinate (K) at (4.00, 8.25);
  \coordinate (L) at (6.49, 8.26);
  \coordinate (M) at (4.25, 8.25);
  \coordinate (N) at (6.00, 8.25);
  \coordinate (O) at (6.49, 7.76);

  % Drawing
  \draw[-{Stealth[scale=1.5]}, shorten >=5pt, shorten <=0pt] (B) -- (C);
  \draw[-{Stealth[scale=1.5]}, shorten >=5pt, shorten <=0pt] (B) -- (D);
  \draw[{Stealth[scale=1.5]}-, shorten >=0pt, shorten <=5pt] (A) -- (B);
  \draw[{Stealth[scale=1.5]}-, shorten >=0pt, shorten <=5pt] (A) -- (E);
  \draw[-{Stealth[scale=1.5]}] (A) -- (F);
  %\draw (D_1) -- (G);
  \draw[{Stealth[scale=1.5]}-, shorten >=0pt, shorten <=5pt] (C) -- (H);
  \draw[{Stealth[scale=1.5]}-, shorten >=0pt, shorten <=5pt] (D) .. controls (5.75, 8.25) and (4.00, 8.25) .. (I);
  \draw[-{Stealth[scale=1.5]}] (C) .. controls (4.25, 8.25) and (6.00, 8.25) .. (L);
  \draw[-{Stealth[scale=1.5]}] (D) -- (O);
  \fill[black] (A) circle (3pt);
  \fill[black] (B) circle (3pt);
  \fill[black] (C) circle (3pt);
  \fill[black] (D) circle (3pt);
  \node[above, black] at (E) {$z_3$};
  \node[above, black] at (F) {$z'_3$};
  \node[left, black] at (H) {$z_2$};
  \node[left, black] at (I) {$z_1$};
  \node[right, black] at (L) {$z'_1$};
  \node[right, black] at (O) {$z'_2$}; 
  \node at (4.69,7.50) {$v_2$};
  \node at (5.27,7.50) {$v_1$};
  \node at (5.15,6.45) {$v_3$};
\end{tikzpicture}
\caption{An $(r,2)$-NZF on the Flower $6$-pole.}\label{fig:Flower_pole}
\end{center}
\end{figure}
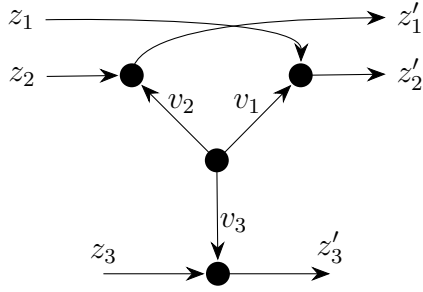

\section{Proof of the lower bound}

\begin{proof}
    In this section, we prove Theorem \ref{thm:lower-bound}. Throughout the proof, we always take indices modulo $3$.
    
 Let $f$ be a %$(2 + c', 2)$-NZF
    $(1+\lambda, 2)$-NZF of the $6$-pole $F$, for $1 < \lambda < \lambda^* = 2\sin\frac{5}{22}\pi$. We denote the flow values of $f$ according to Figure \ref{fig:Flower_pole}. The main idea of our proof is to investigate how the triple $(z_1,z_2,z_3)$ is transformed into the triple $(z'_1,z'_2,z'_3)$. For this, we  define four configurations of triples of vectors, which will be called C1$^+$, C1$^-$, C2$^+$, and C2$^-$, and show that for %$c' < c$ 
    $\lambda < \lambda^*$ the $6$-pole $F$ admits only transitions $\text{C1}^+ \leftrightarrow \text{C1}^-$ and $\text{C2}^+ \leftrightarrow \text{C2}^-$ making it impossible to have an odd number of copies of $F$ in $J_n$. To do so, we compute possible values of the angles between the vectors used in $f$ depending on $\lambda$. The four configurations and the value of $\lambda^*$ are defined in such a way that no triple of vectors can be in more than one configuration. 

 If $u_1$, $u_2$, and $u_3$ are three vectors with zero sum, then we can arrange them to form the triangle $T = T(u_1, u_2, u_3)$ by placing the tail of $u_2$ on the head of $u_1$ and the tail of $u_3$ on the head of $u_2$. If all the vectors $u_1$, $u_2$, and $u_3$ are values of the flow $f$, the sides of the triangle $T$ have lengths from the interval $[1, \lambda] \subset [1, \lambda^*)$, so the internal angles of $T$ are from the interval $(\alpha, \beta)$, where
    $$\alpha = 2\arcsin\left(\frac{1}{2\lambda^*}\right) \qquad  \text{and}\qquad \beta = 2\arcsin\left(\frac{\lambda^*}{2}\right) = \frac{5}{11}\pi,$$ see Figure \ref{fig:alpha_beta}.

 \begin{figure}[h]
  \centering

\begin{tikzpicture}[
    dot/.style={circle, fill=black, inner sep=2pt}, % Stile per i vertici
    % line/.style={thick, color=black!70!brown},     % Stile e colore delle linee dei triangoli
    scale=2 % Scala generale per una migliore visualizzazione
]

\coordinate (T1_A) at (0, 0);       % Base sinistra
\coordinate (T1_B) at (1, 0);       % Base destra (lunghezza base = 1 unità)
\coordinate (T1_Top) at (0.5, 1.3737); % Vertice calcolato per un angolo di 40°

\node at (0.1,0.65) {$\lambda$};
\node at (0.9,0.65) {$\lambda$};
\node at (0.5,-0.15) {1};

% Disegna e colora l'angolo α
\pic [draw, line width=1pt, red, fill=angleRed, angle radius=0.6cm, %"$\alpha$",
text=red] 
    {angle = T1_A--T1_Top--T1_B};

\node at (0.5, 0.97) {\textcolor{red}{$\alpha$}};

\coordinate (T2_Origin) at (1.8, 0);

\coordinate (T2_A) at (1.8, 0);       
\coordinate (T2_B) at (3.8, 0);       
\coordinate (T2_Top) at (2.8, 0.8391); 

\node at (2.2,0.5) {$1$};
\node at (3.4,0.5) {$1$};
\node at (2.8,-0.15) {$\lambda$};

% Disegna e colora l'angolo β
\pic [draw, line width=1pt, green!50!black, fill=angleGreen, angle radius=0.5cm, %"$\beta$",
text=green!50!black, pic text options={yshift=-4pt}] 
    {angle = T2_A--T2_Top--T2_B};

\node at (2.8, 0.45) {\textcolor{green!50!black}{$\beta$}};

% Disegna il triangolo e i punti sui vertici
\draw (T1_A) -- (T1_B) -- (T1_Top) -- cycle;
\draw (T2_A) -- (T2_B) -- (T2_Top) -- cycle;

\end{tikzpicture}

\caption{The angles $\alpha$ and $\beta$.}\label{fig:alpha_beta}
\end{figure}
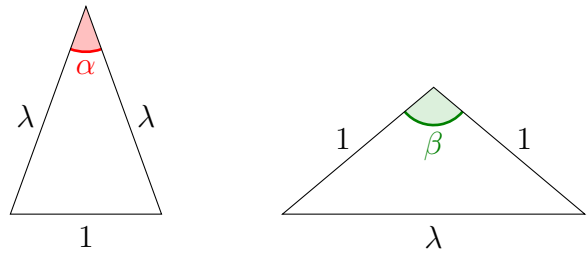
 
 We consider the triangle $T(w_1, w_3, w_2)$ at a fixed place, where $(w_1, w_2, w_3) \in \{(v_1, v_2, v_3), (v_1, v_3, v_2)\}$ is chosen in such a way that the triangle $T(w_1, w_2, w_3)$ is oriented clockwise; see Figure \ref{fig:triangles}. For each $i \in  \{1, 2, 3\}$, we also consider the one of two triangles $T(v_i, z_i, -v_i - z_i)$ and $T(v_i, -v_i - z_i, z_i)$ which intersect $T(w_1, w_3, w_2)$ only on its side $v_i$ (which is one of $w_1, w_2, w_3$).
 We obtain the configuration of the vectors $w_i$, $z_i$ and $-z_i'$ for $i \in \{1, 2, 3\}$. With reference to Figure \ref{fig:triangles}, the vectors $z_1$, $z_2$ and $z_3$ are three vectors from $\{a_1,a_2,a_3,a_1',a_2',a_3'\}$ such that no two of them are in the same triangle, and $-z_1'$, $-z_2'$, and $-z_3'$ are the remaining three vectors and they share a common triangle with $z_2$, $z_1$ and $z_3$, respectively.
  
 \begin{figure}[h]
 \centering
 \includegraphics[]{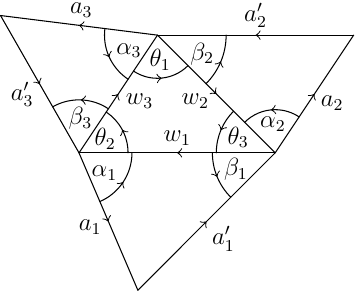}
 \caption{Configuration of four triangles obtained from the flow values of $F$ in a $(1+\lambda, 2)$-NZF with $1 < \lambda < \lambda^* (= 2\sin\frac{5}{22}\pi)$.}
 \label{fig:triangles}
 \end{figure}

 Let us denote the internal angles of the four triangles as follows $\theta_i = \angle(-w_{i-1}, w_{i+1})$, $\alpha_i = \angle(a_i, -w_i)$, and $\beta_i = \angle(w_i, -a_i')$ for $i \in \{1, 2, 3\}$. Recall that all of these angles are from the interval $(\alpha, \beta)$.
    Using Lemma \ref{lemma:two_angles}, we estimate angles between the following vectors:
 \begin{equation}\label{eq:angle1}
 \begin{aligned}
 \angle(a_i, a_{i+1}) &= \pi + \alpha_i - \theta_{i-1} - \alpha_{i+1} \\
 &\in \pi + \left(\tfrac\pi2 - \tfrac32\beta, \tfrac32\beta - \tfrac\pi2\right) + (-\beta, -\alpha)\\
 &\in \left(\tfrac32\pi - \tfrac52\beta, \tfrac\pi2 - \alpha + \tfrac32\beta\right) \subset \left(\tfrac{4}{11}\pi, \pi\right),
 \end{aligned}
 \end{equation}
 where one can easily verify that $\frac\pi2 - \alpha + \frac32\beta < \pi$; and, similarly,
 \begin{equation}\label{eq:angle2}
 \begin{aligned}
 \angle(a_i', a_{i+1}') &= \pi - \beta_i - \theta_{i-1} + \beta_{i+1}\\
 &\in \pi + (-\beta, -\alpha) + \left(\tfrac\pi2 - \tfrac32\beta, \tfrac32\beta - \tfrac\pi2\right)\\
 &\in \left(\tfrac32\pi - \tfrac52\beta, \tfrac\pi2 - \alpha + \tfrac32\beta\right) \subset \left(\tfrac{4}{11}\pi, \pi\right).
 \end{aligned}
 \end{equation}
    Furthermore,
 \begin{equation}\label{eq:angle3}
  \angle(a_i', a_{i+1}) = \pi - \beta_i - \theta_{i-1} - \alpha_{i + 1} \in (\pi - 3\beta, \pi - 3\alpha) \subset \left( -\tfrac{4}{11}\pi, \tfrac{3}{11}\pi \right) ,
 \end{equation}
 easily verifying that $\pi - 3\alpha < \frac{3}{11}\pi$;
 and finally,
 
 \begin{equation}\label{eq:angle4}
  \begin{aligned}
  \angle(a_{i+2}, a_{i}') &= -\pi + \alpha_{i+2} - \theta_{i+1} + \beta_i\\
  &\in -\pi + (\alpha, \beta) + \left( \tfrac\pi2 - \tfrac32\beta, \tfrac32\beta - \tfrac\pi2 \right)\\
  &\in \left(-\tfrac\pi2 + \alpha -\tfrac32\beta, \tfrac52\beta -\tfrac32\pi \right) \subset \left(-\pi, -\tfrac{4}{11}\pi \right).
  \end{aligned}
 \end{equation}

    According to these estimations, we call two vectors \emph{close} if they subtend an angle whose absolute value is smaller than  $\tfrac{4}{11}\pi$. We see that the vectors $a_i'$ and $a_{i + 1}$ are the only pairs of close vectors.
    Therefore, for a triple of vectors $(u_1, u_2, u_3)$, we will consider the following two configurations according to whether the triple contains close vectors:
 \begin{itemize}
 \item[C1] The triple contains no pair of close vectors. Moreover, we use C1$^+$ if we meet the vectors $u_1$, $u_2$, and $u_3$ along the unit circle in the positive (anti-clockwise) direction; otherwise, we use C1$^-$.
 \item[C2] Vectors $u_i$ and $u_j$ for some $i \ne j \in \{1, 2, 3\}$ are close. Then, we call the remaining vector $u_k \notin \{u_i, u_j\}$ \emph{lonely}. If at least one of the angles $\angle(u_k, u_i)$ and $\angle(u_k, u_j)$ is positive, we use the notation C2$^+$; and if at least one of the angles is negative, we use C2$^-$.
 \end{itemize}
    By definition, any triple of vectors is in exactly one of Configurations C1$^+$, C1$^-$, and C2. However, a triple of vectors can be in both Configurations C2$^+$ and C2$^-$.
    Note that the configurations of $(u_1, u_2, u_3)$ and $(-u_1, -u_2, -u_3)$ are the same, since $\angle(u, v) = \angle(-u, -v)$.
 
 Now, we describe the possible transitions between the triples of vectors $(z_1, z_2, z_3)$ and $(z_1', z_2', z_3')$ made by the $6$-pole $F$.

We distinguish two possible cases. Either the vectors $z_1$, $z_2$, and $z_3$ do not meet in any vertex (see Figure \ref{fig:triagnles-c1} for an example), or exactly two of them meet at a vertex (see Figure \ref{fig:triagnles-c2} for an example). Without loss of generality, these cases can be summarised as follows. 
 
 \begin{figure}[h]
 \centering
 \begin{subfigure}{0.45\linewidth}
 \centering
 \includegraphics[]{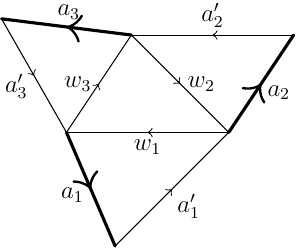}
 \caption{Case 1.}
 \label{fig:triagnles-c1}
 \end{subfigure}
    \quad
 \begin{subfigure}{0.45\linewidth}
 \centering
 \includegraphics[]{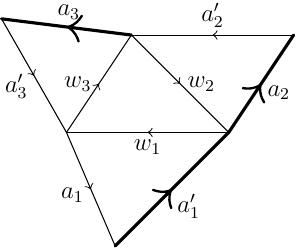}
 \caption{Case 2.}
 \label{fig:triagnles-c2}
 \end{subfigure}
    \caption{Two possible configurations for the vectors $z_1$, $z_2$, $z_3$.}
 \end{figure} 
 
 \paragraph{Case 1}
 We have $(z_1, z_2, z_3) = (a_{\sigma(1)}, a_{\sigma(2)}, a_{\sigma(3)})$ for some permutation $\sigma \in S_3$,
 see Figure \ref{fig:triagnles-c1}.
    By Estimations (\ref{eq:angle1}) and (\ref{eq:angle2}), we have $$\angle(a_i, a_{i+1}), \angle(a_i', a_{i+1}') \subset (\tfrac{4}{11}\pi, \pi).$$
 Thus, both the triples $(a_1, a_2, a_3)$ and $(a_1', a_2', a_3')$ are in Configuration C1$^+$. So the vectors $z_1, z_2, z_3$ are in Configuration C1$^s$ where $s = +$ or $s = -$ if $\sigma$ is an even or an odd permutation, respectively. Furthermore,
 the triple $(-a_{\sigma(1)}', -a_{\sigma(2)}', -a_{\sigma(3)}')=(z_1+v_1, z_2 + v_2, z_3 + v_3) = (z_2', z_1', z_3')$ is also in configuration C1$^s$.
 
 The triple $(z_1', z_2', z_3')$ is obtained from $(z_2', z_1', z_3')$ by swapping $z_1'$ and $z_2'$ so it is in Configuration C1$^{-s}$. Hence, in this case, the $6$-pole $F$ always switches configuration C1$^+$ to C1$^-$ and vice versa.
 
 \paragraph{Case 2} The triple $(z_1, z_2, z_3)$ is a permutation of $(a_i', a_{i+1}, a_{i+2})$ for some $i \in \{1, 2, 3\}$, and the triple $(z_1', z_2', z_3')$ consists of the remaining three vectors; see Figure \ref{fig:triagnles-c2}.
    By Estimations (\ref{eq:angle3}), (\ref{eq:angle1}), and (\ref{eq:angle4}), we have
    \begin{align*}
        \angle(a_i', a_{i+1}) &\in \left( -\tfrac{4}{11}\pi, \tfrac{3}{11}\pi \right),\\
        % \angle(a_{i + 1}, a_{i + 2}) &\in (\tfrac{4}{11}\pi, \pi), \text{ and}\\
        \angle(a_{i + 2}, a_{i + 1}) &\in (-\pi, -\tfrac{4}{11}\pi), \text{ and}\\
        \angle(a_{i + 2}, a_i') &\in (-\pi, -\tfrac{4}{11}\pi), 
    \end{align*}
    respectively.
 Thus, $a_i'$ and $a_{i + 1}$ are the only close vectors with $\angle(a_{i+2}, a_{i+1})$ and $\angle(a_{i+2}, a_i')$ negative, so the triple $(a_i', a_{i+1}, a_{i+2})$ is in Configuration C2$^-$ with the lonely vector $a_{i+2}$. The vectors $(z_1, z_2, z_3)$ are in Configuration C2$^-$, since permuting the vectors will keep the triple in Configuration C2$^-$.
    Also, by Estimations (\ref{eq:angle3}), (\ref{eq:angle4}), and (\ref{eq:angle1}), we have
    \begin{align*}
        \angle(a_{i+2}', a_{i}) &\in \left( -\tfrac{4}{11}\pi, \tfrac{3}{11}\pi \right),\\
        % \angle(a_{i}, a_{i + 1}') &\in \left(-\pi, -\tfrac{4}{11}\pi \right), \text{ and}\\
        \angle(a_{i + 1}', a_{i}) &\in \left(\tfrac{4}{11}\pi, \pi \right), \text{ and}\\
        \angle(a_{i + 1}', a_{i+2}') &\in (\tfrac{4}{11}\pi, \pi), 
    \end{align*}
    respectively. Thus, the triple $(a_{i+2}', a_i, a_{i+1}')$ as so as $(z_1', z_2', z_3')$ is in Configuration C2$^+$.
 
 Therefore, this case produces the transition $\text{C2}^- \rightarrow \text{C2}^+$. By symmetry, we also have the transition $\text{C2}^+ \rightarrow \text{C2}^-$.
 
 \medskip

 Therefore, if $f$ is a $(1+\lambda, 2)$-NZF of $J_n$ for $\lambda < \lambda^*$, then for every $i \in \{1, \dots, n\}$, the triplets $(z_1,z_2,z_3)$ and $(z'_1,z'_2,z'_3)$ are both of the same configuration type, i.e.\ either C1 or C2, but with different signature. Thus, it is not possible to obtain the same configuration after an odd number of transitions, and so $\phi_2(J_n) \ge 1+\lambda^* = 1 + 2 \sin\frac{5}{22}\pi$, for each odd $n \ge 3$.
\end{proof}

\section{Conclusion}

We conclude our paper by showing that the $2$-dimensional flow number of flower snarks is nonincreasing. It is sufficient to prove that $\phi_2(J_n) \ge \phi_2(J_{n + 2})$ for any $n \ge 3$. We show that this is true in a more general setting --  for any type of flow and also for even $n$.

\begin{lemma}
    For any integer $n \ge 3$, if the graph $J_n$ admits a flow $f$, then the graph $J_{n + 2}$ admits a flow using the same flow values as in the flow $f$.
\end{lemma}

\begin{proof}
    The graph $J_{n + 2}$ can be constructed from the graph $J_n$ by replacing the $6$-pole $F$ with the $6$-pole $F_3$ that consists of three serially joined copies of the $6$-pole $F$. As depicted in Figure \ref{fig:6pole_extension}, we can extend every flow on $F$ to the flow $f'$ on $F_3$ following the notation of the flow values on $F$ from Figure \ref{fig:Flower_pole}. The new flow $f'$ uses only flow values of $f$. Since all the dangling edges of $F_3$ have the same flow values as the corresponding dangling edges in $F$, the flow $f'$ on $F_3$ can be extended to a flow on the entire graph $J_{n + 2}$.
\end{proof}

\begin{figure}[h]
    \centering
\begin{tikzpicture}[
  scale=1.5,
  vertex/.style={circle, fill=black, minimum size=6pt, inner sep=0pt},
  edge/.style={draw, -{Stealth[scale=1.5]}}
]

  % Define a reusable macro to draw a single subgraph unit
  % #1 is the label suffix (1, 2, or 3) to keep node names unique
  \newcommand{\drawunit}[1]{
    \node[vertex] (B#1) at (0, 7) {};
    \node[vertex] (A#1) at (0, 6) {};
    \node[vertex] (C#1) at (-0.866, 7.5) {};
    \node[vertex] (D#1) at (0.866, 7.5) {};
  }

  % Tile the 3 units horizontally using tighter scopes (2.5cm instead of 3.5cm)
  \begin{scope}[xshift=0cm]
    \drawunit{1}
    \draw[edge] (B1) -- (C1) node[midway, below left] {$v_2$};
    \draw[edge] (B1) -- (D1) node[midway, below right] {$v_1$};
    \draw[edge] (B1) -- (A1) node[midway, right] {$v_3$};
  \end{scope}
  
  \begin{scope}[xshift=2.5cm]
    \drawunit{2}
    \draw[edge] (C2) -- (B2) node[midway, below left] {$v_1$};
    \draw[edge] (D2) -- (B2) node[midway, below right] {$v_2$};
    \draw[edge] (A2) -- (B2) node[midway, right] {$v_3$};
  \end{scope}
  
  \begin{scope}[xshift=5.0cm]
    \drawunit{3}
    \draw[edge] (B3) -- (C3) node[midway, below left] {$v_2$};
    \draw[edge] (B3) -- (D3) node[midway, below right] {$v_1$};
    \draw[edge] (B3) -- (A3) node[midway, right] {$v_3$};
  \end{scope}

  % ---------------------------------------------------------
  % External Coordinates for Inputs (Left) and Outputs (Right)
  % ---------------------------------------------------------
  \coordinate (InZ3) at (-1.5, 6);
  \coordinate (InZ2) at (-1.5, 7.5);
  \coordinate (InZ1) at (-1.5, 8.0); % Lowered slightly for flatter curve

  \coordinate (OutZ3) at (6.5, 6);     % Brought inward to match tighter spacing
  \coordinate (OutZ2) at (6.5, 7.5);
  \coordinate (OutZ1) at (6.5, 8.0);

  % ---------------------------------------------------------
  % Connections mapping to your hand-drawn sketch
  % ---------------------------------------------------------

  % 1. Bottom straight continuous line
  \draw[edge] (InZ3) -- (A1) node[at start, left] {$z_3$};
  \draw[edge] (A1) -- (A2) node[midway, below] {$z'_3$};
  \draw[edge] (A2) -- (A3) node[midway, below] {$z_3$};
  \draw[edge] (A3) -- (OutZ3) node[at end, right] {$z'_3$};

  % 2. Middle straight connections
  \draw[edge] (InZ2) -- (C1) node[at start, left] {$z_2$};
  \draw[edge] (D1) -- (C2) node[midway, below] {$z'_2$};
  \draw[edge] (D2) -- (C3) node[midway, below] {$z_2$};
  \draw[edge] (D3) -- (OutZ2) node[at end, right] {$z'_2$};

  % 3. Top curved connections (Flatter arcs with 30/150 angles)
  % Incoming z1 arches into D1
  \draw[edge] (InZ1) to[out=0, in=150] node[at start, left] {$z_1$} (D1);

  % C1 arches over to D2
  \draw[edge] (C1) to[out=30, in=150] node[midway, above] {$z'_1$} (D2);

  % C2 arches over to D3
  \draw[edge] (C2) to[out=30, in=150] node[midway, above] {$z_1$} (D3);

  % C3 arches out to the z'_1 exit
  \draw[edge] (C3) to[out=30, in=180] node[at end, right] {$z'_1$} (OutZ1);

\end{tikzpicture}
    \caption{Extension of the flow on $F$ to three copies of $F$}
    \label{fig:6pole_extension}
\end{figure}
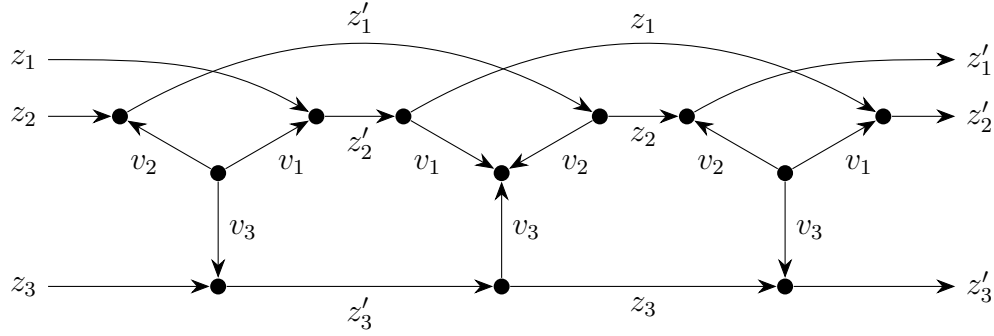

\begin{corollary}\label{cor:nonincreasing}
    The sequence $(\phi_2(J_{2k + 1}))_{k \ge 1}$ of the flow numbers of the flower snarks is nonincreasing.
\end{corollary}

Combining Corollary \ref{cor:nonincreasing} with the numerical construction of a $2$-dimensional flow on $J_5$ from \cite{MMRT_d-dim-flows}, we conclude that for each odd $n \ge 5$,
$$\phi_2(J_n) \in [1 + 2 \sin\frac{5}{22}\pi, 2.387893647] \subset [2.309721467, 2.387893647].$$

Unfortunately, we lack an exact description of the flow on $J_5$. 
Figure \ref{fig:transitions} illustrates transitions through each of the five $6$-poles $F$ around $J_5$. For each $i \in \{1, 2, 3, 4, 5\}$, the vectors $v_1^i$, $v_2^i$, and $v_3^i$ are depicted black with their numbers. The triple $(z_1^i, z_2^i, z_3^i)$ has a unique color for each $i$.

Note that, in this case, the triples of vectors $(z_1^i, z_2^i, z_3^i)$ alternate between Configurations $C2^-$ and $C2^+$ for $i \in \{2, 3, 4, 5\}$, while the triple $(z_1^1, z_2^1, z_3^1)$ is in Configuration $C1^-$. The angles between these three vectors, $\angle(z_3^1, z_2^1) \approx 0.3736\pi$ and $\angle(z_2^1, z_1^1) \approx 0.3742\pi$, are only slightly grater than $\frac{4}{11}\pi \approx 0.3636$.

\begin{figure}
    \centering
    \includegraphics[]{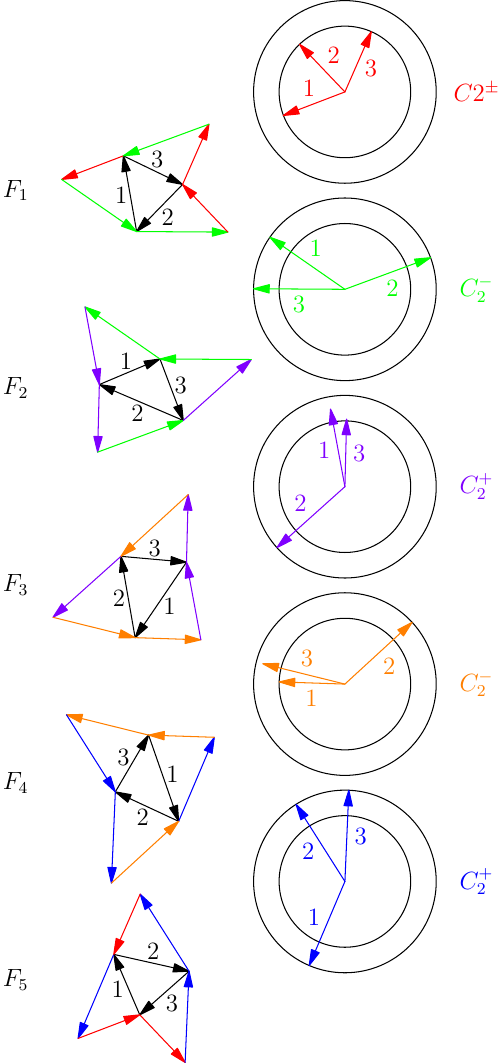}
    \caption{Transitions through each of the $6$-poles $F$ in $J_5$ determined by the flow constructed in \cite{MMRT_d-dim-flows}.}
    \label{fig:transitions}
\end{figure}

\section*{Acknowledgments}

We thank Giuseppe Mazzuoccolo and Gloria Tabarelli for fruitful discussions on the topic. The second author acknowledges partial support from the research grants APVV-23-0076, VEGA 1/0173/25, and VEGA 1/0613/26.

\end{document}